%% file: FX_v1.tex
\documentclass[10pt]{article}

\usepackage{amsmath,amssymb,amscd,amsthm,amsfonts}
\usepackage{graphicx,subfigure}
\usepackage{hyperref}
\usepackage{dsfont}
\usepackage{color}
\usepackage{float}

\newtheorem{theorem}{Theorem}[section]

\newtheorem{lemma}[theorem]{Lemma}

\newtheorem{corollary}[theorem]{Corollary}

\newtheorem{definition}[theorem]{Definition}
\newtheorem{remark}[theorem]{Remark}

\title{Stable parabolic Higgs bundles of rank two and singular hyperbolic metrics}

\author{Yu Feng \and Bin Xu}
\date{}

\begin{document}

\maketitle
\begin{abstract}
In this paper, we construct a stable parabolic Higgs bundle of rank two, which corresponds to the uniformization associated with a conformal hyperbolic metric on a compact Riemann surface $\overline{X}$ with prescribed singularities. This provides an alternative proof of the classical existence theorem for singular hyperbolic metrics, originally established by Heins ({\it Nagoya Math. J.} 21 (1962), 1-60). We also introduce a family of stable parabolic Higgs bundles of rank two on $\overline{X}$, parametrized by a nonempty open subset of a complex vector space. These bundles correspond to singular hyperbolic metrics with the same type of singularity as the original, but are defined on deformed Riemann surfaces of $\overline{X}$. Thus, we extend partially the final section of Hitchin's celebrated work ({\it Proc. London Math. Soc.} 55(3) (1987), 59-125) to the context of hyperbolic metrics with singularities.
\end{abstract}

\section{Introduction}\label{section-introduction}

In his seminal paper, Hitchin \cite{Hitchin:1987} constructed a specific stable Higgs bundle of rank two on a compact Riemann surface of genus greater than one, and solved the self-duality equations to obtain a conformal hyperbolic metric on the surface. Furthermore, by deforming the Higgs field, he recovered all constant Gaussian curvature $-1$ Riemannian metrics on the underlying topological surface. Subsequent studies have partially extended these results to hyperbolic metrics with singularities. Biswas, Bradlow, Dumitrescu, and Heller \cite{BBDH:2021} studied hyperbolic conical metrics on compact Riemann surfaces of genus greater than one, with cone angles in $2\pi\mathbb{Z}_{>1}$, using stable ${\rm SL}(2, \mathbb{R})$-Higgs bundles. Biswas, Ar{\'e}s-Gastesi, and Govindarajan \cite{BGG:1997} employed parabolic Higgs bundles to construct complete hyperbolic metrics on punctured surfaces, where the punctures correspond to cusp singularities of the metrics. Kim and Wilkin \cite{KW:2018} investigated hyperbolic cone metrics with cone angles in $(0, 2\pi)$ on compact Riemann surfaces, also using stable parabolic Higgs bundles of rank two. Additional relevant studies include, but are not limited to, \cite{Baptista:2014}, \cite{BB:2013}, \cite{Mondello:2016}, and \cite{NS:1995}.

In this manuscript, inspired by the works of \cite{Hitchin:1987, BGG:1997, KW:2018, BBDH:2021}, we provide a more comprehensive generalization of Hitchin's results to include a broader class of singular hyperbolic metrics, incorporating both cone and cusp singularities. Specifically, we construct a stable parabolic Higgs bundle of rank two on a compact Riemann surface (Lemma \ref{lemma:construction of parabolic Higgs bundle}), solve Hitchin's equation on it, and obtain a conformal hyperbolic metric with prescribed singularities on the Riemann surface from the Hermitian-Einstein metric on the Higgs bundle (Subsection \ref{subsec:thm1.1}). Moreover, by deforming the Higgs field, we derive a family of parabolic Higgs bundles on a compact Riemann surface, parametrized by a nonempty open subset of a complex vector space (Theorem \ref{theorem:main theorem}). These bundles correspond to Riemannian hyperbolic metrics with the same singularity type as the original but are conformal on deformed Riemann surfaces. Additionally, we show that all Riemannian hyperbolic metrics with cone singularities of cone angles in $(0, \pi]$ can be recovered by deforming the Higgs field (Corollary \ref{corollary:isometric}).


\subsection{Background of singular hyperbolic metrics}
Let $\overline{X}$ be a compact Riemann surface with non-negative genus, and let
$\mathrm{\hat{D}} = \sum _{i=1}^{m}(\theta _{i} - 1)p_{i}$ be an ${\mathbb R}$-divisor
on $\overline{X}$ such that $0 \leq \theta _{i} \neq 1$, where
$\{p_{i}\}_{i=1}^{m}$ are $m$ distinct points on $\overline{X}$. We define $X := \overline{X} - \{p_{i}\}_{i=1}^{m}$, and let $j: X \to \overline{X}$ denote the inclusion map.
We say that $\mathrm{d} s^{2}$ is a {\it singular hyperbolic metric representing
$\mathrm{\hat{D}}$} if the following conditions hold:
\begin{itemize}
    \item $\mathrm{d} s^{2}$ is a conformal metric on $X$ with Gaussian curvature
    $-1$.
    \item For $\theta _{i} > 0$, $\mathrm{d} s^{2}$ has a {\it cone singularity
    at $p_{i}$ with a cone angle of $2\pi \theta _{i}$}. Specifically, in a neighborhood
    $U$ of $p_{i}$, $\mathrm{d} s^{2} = e^{2u} | \mathrm{d} z |^{2}$, where $z$ is a
    complex coordinate of $U$ with $z(p_{i}) = 0$, and $u - (\theta _{i} - 1) \ln |z|$ extends continuously to $z = 0$.
    \item For $\theta _{i} = 0$, $\mathrm{d} s^{2}$ has a {\it cusp singularity
    at $p_{i}$}. Specifically, in a neighborhood $V$ of $p_{i}$, $\mathrm{d} s^{2} = e^{2u} | \mathrm{d} z |^{2}$, where $z$ is a complex coordinate of $V$ with $z(p_{i}) = 0$, and $u + \ln |z| + \ln (-\ln |z|)$ extends continuously to $z = 0$.
\end{itemize}

Nitsche \cite{Nitsche:1957} and Heins \cite{Heins:1962} showed that an isolated singularity of a conformal hyperbolic metric can only be a cone singularity or a cusp singularity. Li and Mochizuki provided a Higgs bundle proof for this result \cite[Lemma 8.3]{LM:2020}. See also \cite{FSX:2020} for an elementary proof of this result using Complex Analysis.
This is a classical problem concerning the existence and uniqueness of a hyperbolic metric with finitely many prescribed singularities on a compact Riemann surface. According to the Gauss-Bonnet formula, if there is a hyperbolic
metric representing the divisor $\mathrm{\hat{D}} = \sum _{i=1}^{m}(\theta _{i} - 1)p_{i}$,
where $0 \leq \theta _{i} \neq 1$, on a compact Riemann surface $\overline{X}$, then the following inequality holds:
$$ \chi (\overline{X}) + \sum _{i=1}^{m} (\theta _{i} - 1) < 0, $$
where $\chi (\overline{X})$ is the Euler characteristic of $\overline{X}$. Over half a century ago, Heins \cite{Heins:1962} studied the properties of S-K metrics and used them to prove the following result:

\begin{theorem}{\rm (Heins \cite{Heins:1962})}
\label{thm:rat}
There exists a hyperbolic metric representing an $\mathbb{R}$-divisor
$\mathrm{\hat{D}} = \sum _{i=1}^{m} (\theta _{i} - 1) p_{i}$ with
$0 \leq \theta _{i} \neq 1$ on a compact Riemann surface $\overline{X}$ if and only
if $\chi (\overline{X}) + \sum _{i=1}^{m} (\theta _{i} - 1) < 0$. Moreover, such a metric is unique.
\end{theorem}

\subsection{Main results}

Our main results are divided into two parts. The first part revisits Heins' result (Theorem \ref{thm:rat}) from the perspective of stable parabolic Higgs bundles. To begin, we first establish the following key lemma. 

Let $\overline{X}$ be a compact Riemann surface, and let $X = \overline{X} - \{p_1, \dots, p_m\}$. Let $K$ denote the canonical line bundle of $\overline{X}$, and let $K^{\frac{1}{2}}$ be a line bundle such that $(K^{\frac{1}{2}})^2 = K$. Note that $K^{\frac{1}{2}}$ has $2^{2-\chi(\overline{X})}$ different choices.

\begin{lemma}
\label{lemma:construction of parabolic Higgs bundle}
Given that $0 \leq \theta_i \neq 1$ for each $p_i$, and assuming that $\chi(\overline{X}) + \sum_{i=1}^{m} (\theta_i - 1) < 0$, let $D = p_1 + \cdots + p_m$ denote the reduced divisor. Define $\tilde{D} = [\theta_1]p_1 + \cdots + [\theta_m]p_m$, where $[x]$ denotes the Gauss rounding function for a real number $x$, and $\{x\}:= x - [x]$. 

Next, we introduce a parabolic structure on the vector bundle 
\[
E_0 := K^{\frac{1}{2}} \oplus \left(K^{-\frac{1}{2}} \otimes \mathcal{O}_{\overline{X}}(\tilde{D} - D)\right)
\]
at each $p_i$ as follows:

\begin{itemize}
\item Case 1: $\{\theta_i\} \neq 0$
\begin{itemize}
    \item[(i)] A flag $(E_{0})_{p_i} \supset \big(K^{-\frac{1}{2}} \otimes \mathcal{O}_{\overline{X}}(\tilde{D} - D)\big)_{p_{i}} \supset 0$,
    \item[(ii)] Weights $0 < \alpha_1 < \alpha_2 < 1$, where $\alpha_1 = \frac{1}{2}(1 - \{\theta_i\})$ and $\alpha_2 = \frac{1}{2}(1 + \{\theta_i\})$.
\end{itemize}
\item Case 2: $\{\theta_i\} = 0$
\begin{itemize}
    \item[(i)] A trivial flag $(E_{0})_{p_i} \supset 0$,
    \item[(ii)] A weight of $\frac{1}{2}$ at $(E_{0})_{p_i}$.
\end{itemize}
\end{itemize}

We now define a parabolic Higgs field
\[
\phi = \frac{1}{2} \left( \begin{array}{cc}
0 & 0 \\
1_{\tilde{D}} & 0 \\
\end{array} \right) \in H^0\left({\rm End}(E_0) \otimes K \otimes \mathcal{O}_{\overline{X}}(D)\right),
\]
where $1_{\tilde{D}}$ denotes the natural section of $\mathcal{O}_{\overline{X}}(\tilde{D})$ corresponding to the constant function $1$ on $\overline{X}$, so that $\tilde{D}$ is the zero divisor of $1_{\tilde{D}}$. 

Then, $(E_0, \phi)$, endowed with the preceding parabolic structure, is a stable parabolic Higgs bundle with parabolic degree $\sum_{i=1}^{m} [\theta_i]$.
\end{lemma}

Using the stable parabolic Higgs bundle constructed in Lemma \ref{lemma:construction of parabolic Higgs bundle}, we conclude the first part of our main results by providing an alternative proof of Theorem \ref{thm:rat} in Subsection \ref{subsec:thm1.1}.

In the second part, we deform the Higgs field $\phi$ described above and derive a family of stable parabolic Higgs bundles on $\overline{X}$. These bundles are parameterized by a nonempty open subset of a complex vector space of meromorphic quadratic differentials on $\overline{X}$ with poles at most along a divisor. They correspond to singular hyperbolic metrics with the same type of singularities as the original metric. These metrics are defined on a family of deformed Riemann surfaces of $\overline{X}$. Before stating our main theorem (Theorem \ref{theorem:main theorem}), we first introduce some necessary notions.

First we recall from the proof of Theorem \ref{thm:rat} in Section \ref{section:Proofs of the key Lemma} that we obtain a singular hyperbolic metric $\widehat{{\rm d}s^{2}}$ representing the divisor $\hat{D} = \sum_{i=1}^{m}(\theta_i - 1)p_i$ from the stable parabolic Higgs bundle $(E_0, \phi)$. We denote the set $\{p_i\}_{i=1}^{m}$ as ${\rm supp\, D}$. Let ${\rm supp\, D_0} = \{p_i \in {\rm supp\, D} : \theta_i = 0\}$, which corresponds to the cusp singularities of $\widehat{{\rm d}s^{2}}$. 
Next, we express the set of cone points as a disjoint union:
\[
{\rm supp\, D} \setminus {\rm supp\, D_0} = {\rm supp\, D_1} \sqcup {\rm supp\, D_2} \sqcup {\rm supp\, D_3} \sqcup {\rm supp\, D_4},
\]
where
\[
{\rm supp\, D_1} = \Big\{ p_i \in {\rm supp\, D} \setminus {\rm supp\, D_0} : \{\theta_i\} = 0 \Big\},
\]
\[
{\rm supp\, D_2} = \Big\{ p_i \in {\rm supp\, D} \setminus {\rm supp\, D_0} : 0 < \{\theta_i\} < \frac{1}{2} \Big\},
\]
\[
{\rm supp\, D_3} = \Big\{ p_i \in {\rm supp\, D} \setminus {\rm supp\, D_0} : \{\theta_i\} = \frac{1}{2} \Big\},
\]
and
\[
{\rm supp\, D_4} = \Big\{ p_i \in {\rm supp\, D} \setminus {\rm supp\, D_0} : \frac{1}{2} < \{\theta_i\} < 1 \Big\}.
\]
Let $D_1 = p_{i_1} + \cdots + p_{i_{D_1}}$ denote the reduced divisor corresponding to ${\rm supp\, D_1}$. Similarly, we can define the reduced divisors $D_2$, $D_3$, and $D_4$. Let $1_D$ denote the natural section of $\mathcal{O}_{\overline{X}}(D)$ given by the constant function $1$, so that the zero divisor of $1_D$ is $D$.

We now state our main theorem in the second part.

\begin{theorem}
\label{theorem:main theorem}
For any $a \in H^{0}\left( \overline{X}, K^{2} \otimes \mathcal{O}_{\overline{X}}(D - \tilde{D}) \right)$, the Higgs field
\[
\Phi_{a} := \left( \begin{array}{cc}
0 & a \otimes 1_D \\
\frac{1}{2} 1_{\tilde{D}} & 0 \\
\end{array} \right),
\]
on the parabolic bundle $E_0$ (defined in Lemma \ref{lemma:construction of parabolic Higgs bundle}) makes $(E_0, \Phi_a)$ a stable parabolic Higgs bundle with parabolic degree $\sum_{i=1}^{m} [\theta_i]$.

Let $H_a$ denote the Hermitian-Einstein metric on the restriction of $E_0$ to $X$, and let $h_a$ denote the Kähler metric on $X$ induced by $H_a$. Then the following results hold:

1. There exists a unique largest connected subset
\[
\mathcal{U} \subset H^{0}\left( \overline{X}, K^{2} \otimes \mathcal{O}_{\overline{X}}(D - 2\tilde{D} + D_1 - D_4) \right)
\]
containing $0$ such that for all $a \in \mathcal{U}$, the section of the second symmetric power of the complexified cotangent bundle
\[
\hat{h}_a := 2a + h_a + 2\bar{a} + \frac{4a \bar{a}}{h_a} \in \Omega^{0}\left( X, \mathrm{Sym}^{2}(T^{*}_{\mathbb{R}}X \otimes \mathbb{C}) \right)
\]
is a Riemannian metric on $X$. The subset $\mathcal{U}$ is either $\{0\}$ or a nonempty open subset of 
\[
H^{0}\left( \overline{X}, K^{2} \otimes \mathcal{O}_{\overline{X}}(D - 2\tilde{D} + D_1 - D_4) \right).
\] 
In addition, it satisfies
\[
H^{0}\left( \overline{X}, K^{2} \otimes \mathcal{O}_{\overline{X}}(D - 2\tilde{D} - D_3 - D_4) \right) \subset \mathcal{U} \subset H^{0}\left( \overline{X}, K^{2} \otimes \mathcal{O}_{\overline{X}}(D - 2\tilde{D} + D_1 - D_4) \right),
\]
and the corresponding metric at parameter $a=0$ lies  in the conformal class of $X$.

2. The metric $\hat{h}_a$ is a Riemannian metric of constant Gaussian curvature $-1$. Moreover, $\hat{h}_a$ has a cone singularity at $p_i$ with cone angle $2\pi \theta_i$ if $\theta_i \neq 0$, and a cusp singularity at $p_i$ if $\theta_i = 0$. In other words, $\hat{h}_a$ and $h_0$ have the same singularity type.

\end{theorem}

\begin{remark}
Assume that all cone angles $2\pi \theta_i \in 2\pi \mathbb{Z}_{>1}$, i.e. $D_0 = D_2 = D_3 = D_4 = 0$. Then Theorem \ref{theorem:main theorem} recovers the main result of \cite{BBDH:2021}.
\end{remark}

From the proof of Theorem \ref{theorem:main theorem}, we obtain the following corollary:

\begin{corollary}
\label{corollary:restricted angles}
If all $\theta_i \in [0, \frac{1}{2})$, then $D_1 = D_3 = D_4 = 0$. Thus, for any $a \in H^{0}\left( \overline{X}, K^{2} \otimes \mathcal{O}_{\overline{X}}(D) \right)$, $\hat{h}_a$ gives a Riemannian metric of constant Gaussian curvature $-1$.
\end{corollary}

Inspired by Theorem (11.2)(iii) by Hitchin \cite{Hitchin:1987}, we can prove the following:

\begin{corollary}
\label{corollary:isometric}
Assume that $2\pi \theta_i \in (0, \pi]$ for all $1\leq i\leq m$. Then every Riemannian hyperbolic metric on the underlying topological surface of $\overline{X}$ with $m$ cone angles $\{2\pi\theta_i\}_{i=1}^m$  is isometric to a metric of the form $\hat{h}_a$ for some $a \in \mathcal{U}$ in Theorem \ref{theorem:main theorem}. In particular, if  $2\pi \theta_i \in (0, \pi)$ for all $1\leq i\leq m$, 
then we have $\mathcal{U}=H^{0}\left( \overline{X}, K^{2} \otimes \mathcal{O}_{\overline{X}}(D) \right)$.
\end{corollary}

\subsection{Outline}
Section 2 reviews concepts and background related to filtered regular Higgs bundles on punctured Riemann surfaces, which will be used in Section \ref{section:Proofs of the key Lemma}. In Section \ref{section:Proofs of the key Lemma}, we provide the proof of Lemma \ref{lemma:construction of parabolic Higgs bundle}. Using the parabolic Higgs bundle constructed in this lemma, we present a parabolic-Higgs-bundle proof of Theorem \ref{thm:rat}. Finally, in Section \ref{section-universal}, we prove Theorem \ref{theorem:main theorem} and the corollaries \ref{corollary:restricted angles} and \ref{corollary:isometric} by deforming the Higgs field.

\input{section-intersection}

\input{section-colored}

\input{section-universal}

\section{Acknowledgments}
The authors deeply appreciate Mr. Junming Zhang for his invaluable discussions and support throughout this project, particularly his insightful suggestions on Lemma \ref{lemma:construction of parabolic Higgs bundle} and his many constructive recommendations on the manuscript, all of which have greatly enhanced this work.
The authors would like to thank Qiongling Li and Takuro Mochizuki for helpful discussions and comments. B.X. is supported in part by the Project of Stable Support for Youth Team in Basic Research Field, CAS (Grant No. YSBR-001) and NSFC (Grant No. 12271495). Y.F. is supported in part by the National Key R\&D Program of China No. 2022YFA1006600, Fundamental Research Funds for the Central Universities and Nankai Zhide Foundation.
\newcommand{\etalchar}[1]{$^{#1}$}
\providecommand{\bysame}{\leavevmode\hbox to3em{\hrulefill}\thinspace}
\providecommand{\MR}{\relax\ifhmode\unskip\space\fi MR }
\providecommand{\MRhref}[2]{%
  \href{http://www.ams.org/mathscinet-getitem?mr=#1}{#2}
}
\providecommand{\href}[2]{#2}

\smallskip

\noindent
Yu Feng \\
\textsc{
Chern Institute of Mathematics and LPMC\\
Nankai University\\
Tianjin 300071, China}\\
\textit{E-mail address: }\texttt{yuf@nankai.edu.cn}
\medskip

\noindent
\noindent
Bin Xu \\
\textsc{
School of Mathematical Sciences\\
University of Science and Technology of China\\
Hefei 230026, China}\\
\textit{E-mail address: }\texttt{bxu@ustc.edu.cn}
\medskip

\end{document}

%% file: section-intersection.tex
\section{Preliminaries}
\label{section:filtered regular Higgs bunlde}
\quad In this section, we review the concepts related to filtered regular Higgs bundles, using the notations introduced in the introduction.

\begin{definition}{\rm (\cite[pp. 716-717]{Simpson:1990})}
\label{definition:filtered vector bundle}
{\rm
Let $\overline{X}$ be a compact Riemann surface, and let $\{p_i\}_{i=1}^m$ denote a set of $m$ distinct points on $\overline{X}$. A {\it filtered vector bundle} on $X = \overline{X} \backslash \{p_1, \dots, p_m\}$ is an algebraic vector bundle $E \rightarrow X$, together with a one-parameter family of algebraic vector bundles $E_{\alpha} \rightarrow \overline{X}$ indexed by $\alpha \in \mathbb{R}$, such that for each $\alpha$, we have $E = i^{*} E_{\alpha}$, and the following conditions hold:

\begin{itemize}
\item $\{E_{\alpha}\}_{\alpha \in \mathbb{R}}$ is a decreasing family, meaning that $E_{\alpha}$ is a subsheaf of $E_{\beta}$ whenever $\alpha \geq \beta$.
\item $\{E_{\alpha}\}_{\alpha \in \mathbb{R}}$ is left continuous, i.e., for each $\alpha$, there exists $\epsilon' > 0$ such that $E_{\alpha - \epsilon} = E_{\alpha}$ for all $0 \leq \epsilon \leq \epsilon'$.
\item $E_{\alpha + 1} = E_{\alpha} \otimes \mathcal{O}_{\overline{X}}(-D)$ for all $\alpha$.
\end{itemize}
A {\it filtered regular Higgs bundle} $(E, \phi, \{E_{\alpha}\})$ on $X$ consists of a filtered vector bundle $(E, \{E_{\alpha}\})$ along with a section $\phi \in H^{0}\left({\rm End}(E_0) \otimes K \otimes \mathcal{O}_{\overline{X}}(D)\right)$, which preserves the subsheaf $E_{\alpha} \subset E_0$ for every $\alpha \in (0,1]$.
This $\phi$ is called the {\it Higgs field} of the filtered regular Higgs bundle $(E, \phi, \{E_{\alpha}\})$ on $X$.
}
\end{definition}

\begin{definition}
\label{definition:algebraic degree}
{\rm
Given a filtered vector bundle $\left(E, \{E_{\alpha}\}_{\alpha \in \mathbb{R}}\right)$ on $X$, let $E_{p_i,0}$ denote the fiber of $E_0 \rightarrow \overline{X}$ at the point $p_i \in \overline{X}$, and let ${\rm Gr}_{\alpha}(E_{p_i,0})$ be the direct limit of the system $\{E_{p_i, \alpha}/E_{p_i, \beta}\}$ as $\beta \to \alpha^+$. The {\it parabolic degree} of the filtered vector bundle $\{E_{\alpha}\}_{\alpha \in \mathbb{R}}$ is defined as
\[
\deg\left(E, \{E_{\alpha}\}\right) := \deg(E_0) + \sum_{i=1}^{m} \sum_{0 \leq \alpha < 1} \alpha \, {\rm dim}_{\mathbb{C}}\left({\rm Gr}_{\alpha}(E_{p_i,0})\right).
\]
The filtered regular Higgs bundle $(E, \phi, \{E_{\alpha}\})$ is said to be {\it parabolic Higgs stable} (resp. {\it parabolic Higgs semistable}) if for every subbundle $F \subset E$, endowed with the induced filtration $\{F_{\alpha}\}$ and preserved by $\phi$, the inequality
\[
{\rm slope}\left(F, \{F_{\alpha}\}\right):=\frac{\deg\left(F, \{F_{\alpha}\}\right)}{{\rm rk}\,F} < 
(\leq) {\rm slope}\left(E, \{E_{\alpha}\}\right)=\frac{\deg\left(E, \{E_{\alpha}\}\right)}{{\rm rk}\,E}
\]
holds. In this context, $\big(F,\, \phi, \, \{F_{\alpha}\}\big)$ is called a {\it filtered Higgs subbundle} of $(E, \phi, \{E_{\alpha}\})$. 
}
\end{definition}

We now choose a smooth conformal metric $\mathrm{d}s_0^2$ on $\overline{X}$ such that $\mathrm{d}s_0^2$ is Euclidean in some neighborhood of each point $p_i$,  where $\omega$ is the K\"{a}hler form.

\begin{definition}
\label{definition:acceptable metric}
{\rm
Let $E \rightarrow X$ be an algebraic vector bundle equipped with a smooth Hermitian metric $k$, and let $F_k$ denote the curvature of the Chern connection of $k$. The metric $k$ on $E$ is said to be {\it acceptable} if there exists a bound on $F_k$ near each point $p \in \{p_i\}_{i=1}^m$, of the form
\[
|F_k|_{(k, \, \mathrm{d}s_0^2)} \leq f + \frac{C}{r^2 (\log r)^2}
\]
for some $f \in L^q$ with $q > 1$ and some constant $C > 0$, where $r$ is the distance from $p$ with respect to $\mathrm{d}s_0^2$.
}
\end{definition}

The following theorem due to Simpson is crucial for us to prove Theorems \ref{thm:rat} and \ref{theorem:main theorem}, see also \cite[Theorem 12]{BS:2012}.

\begin{theorem}{\rm (\cite[Theorem 6]{Simpson:1990})}
\label{theorem:Hermitian-Einstein metric}
Consider a filtered regular Higgs bundle $\left(E, \phi, \{E_{\alpha}\}\right)$ on $X$, which is parabolic Higgs stable. 
Given an acceptable Hermitian metric $k$ on $E$, there exists a Hermitian metric $h$ that satisfies Hitchin's self-dual equation on $X$:
\[
F_{h}+[\phi,\phi^{*h}]=-2\pi {\rm i}\cdot \frac{\text{slope}(E,\{E_{\alpha}\})}{{\text{vol}}(X)}\ {\rm Id}_{E}\otimes \omega,
\]
where $F_{h}$ is the curvature of the Chern connection for the metric $h$. The metric $h$ is bounded with respect to $k$. That is, there exists a constant $C>0$ depending only on both $h$ and $k$ such that the inequality $C^{-1}k\leq h\leq Ck$ holds uniformly on $X$.
This $h$ is called a {\rm Hermitian-Einstein metric} on the filtered regular Higgs bundle  
$\left(E, \phi, \{E_{\alpha}\}\right)$.
\end{theorem}

%% file: section-colored.tex
\section{Proofs of Lemma \ref{lemma:construction of parabolic Higgs bundle} and Theorem \ref{thm:rat}}
\label{section:Proofs of the key Lemma}

\subsection{Constructing parabolic Higgs stable bundles}
\quad Inspired by the work in \cite{KW:2018}, we use the language of filtered regular Higgs bundles to provide the construction in Lemma \ref{lemma:construction of parabolic Higgs bundle} and prove that this filtered regular Higgs bundle is parabolic Higgs stable.

\begin{proof}[Proof of Lemma \ref{lemma:construction of parabolic Higgs bundle}]
Recall that $D=p_{1}+\cdots+p_{m}$ and $\tilde{D}=[\theta_{1}]p_{1}+\cdots+[\theta_{m}]p_{m}$.
For each $\alpha\in[0,1)$, define the divisors
\[
  D_{1}(\alpha)=\sum _{i=1}^{m}\epsilon _{i} p_{i},\quad \text{where}\ \epsilon_{i} = \begin{cases}
   0 ,& \text{if}\ \alpha\leq \frac{1}{2}(1-\{\theta_{i}\}); \\
   1 ,& \text{if}\ \alpha>\frac{1}{2}(1-\{\theta_{i}\}).
  \end{cases}
\]
\[
  D_{2}(\alpha)=\sum _{i=1}^{m}\epsilon' _{i} p_{i},\quad \text{where}\ \epsilon'_{i} = \begin{cases}
   0 ,& \text{if}\ \alpha\leq \frac{1}{2}(1+\{\theta_{i}\}); \\
   1 ,& \text{if}\ \alpha>\frac{1}{2}(1+\{\theta_{i}\}).
  \end{cases}
\]
Note that $D_{1}(0)=0$ and $D_{2}(0)=0$.
Let $E_{0}=K^{\frac{1}{2}}\oplus \Big(K^{-\frac{1}{2}}\otimes \mathcal{O}_{\overline{X}}(\tilde{D}-D)\Big)$ be a bundle over $\overline{X}$, and let $E$ denote the restriction of $E_0$ to $X$. 
For each $\alpha\in [0,1)$, the associated extension of $E$ across the punctures is given by
\[
E_{\alpha}=\Big(K^{\frac{1}{2}}\otimes\mathcal{O}_{\overline{X}}\big(-D_{1}(\alpha)\big)\Big)\oplus \Big(K^{-\frac{1}{2}}\otimes\mathcal{O}_{\overline{X}}\big(-D_{2}(\alpha)+\tilde{D}-D\big)\Big)
\]
so that $\{E_\alpha\}$ forms a filtered vector bundle on $X$.
In particular, we have $E_{0}=K^{\frac{1}{2}}\oplus \Big(K^{-\frac{1}{2}}\otimes \mathcal{O}_{\overline{X}}(\tilde{D}-D)\Big)$.
Since 
$$\big(K^{\frac{1}{2}}\big )^{-1}\otimes \Big(K^{-\frac{1}{2}}\otimes \mathcal{O}_{\overline{X}}(\tilde{D}-D)\Big) \otimes K\otimes\mathcal{O}_{\overline{X}}(D)=\mathcal{O}_{\overline{X}}(\tilde{D}), $$
we can define the following Higgs field
\[
\phi=\frac{1}{2}\left (\begin{array}{ccc}
0 & 0    \\
1_{\tilde{D}} & 0   \\
\end{array}\right)\in H^{0}\big({\rm End}(E_{0})\otimes K\otimes\mathcal{O}_{\overline{X}}(D) \big)
\]
for the filtered vector bundle $\{E_\alpha\}$, 
where $1_{\tilde{D}}$ denote the natural section given by the constant function $1$.
Therefore, we obtain a filtered regular Higgs bundle $\big(E,\,\phi,\, \{E_\alpha\}  \big)$ with parabolic degree
\begin{align*}
\deg (E,\{E_{\alpha}\})&=\deg E_{0}+\sum _{i=1}^{m}\frac{1}{2}(1-\{\theta_{i}\})+\sum _{i=1}^{m}\frac{1}{2}(1+\{\theta_{i}\})
=\sum _{i=1}^{m}[\theta_{i}].
\end{align*}

Suppose that  $\big(F,\,\phi,\, \{F_\alpha\}  \big)$  is a filtered Higgs subbundle.  Let $v$ be a local smooth section of $F_{0}$. Then $v=fv_{1}+gv_{2}$, where $v_{1}\in \Gamma\big(K^{\frac{1}{2}}\big)$ and $v_{2}\in \Gamma\big(K^{-\frac{1}{2}}\otimes \mathcal{O}_{\overline{X}}(\tilde{D}-D)\big)$. Since $\phi(v)\subset F_{0}\otimes K\otimes\mathcal{O}_{\overline{X}}(D)$, we have $F_{0}=K^{-\frac{1}{2}}\otimes \mathcal{O}_{\overline{X}}(\tilde{D}-D)$.

The only filtered Higgs subbundle of $\big(E,\,\phi,\, \{E_\alpha\}  \big)$  is 
$$
\Big(K^{-\frac{1}{2}}\otimes \mathcal{O}_{\overline{X}}(\tilde{D}-D),\big\{K^{-\frac{1}{2}}\otimes\mathcal{O}_{\overline{X}}\big(-D_{2}(\alpha)+\tilde{D}-D\big)\big\}_{\alpha}\Big),$$ 
which has the parabolic degree of
\begin{align*}
&\deg\Big(K^{-\frac{1}{2}}\otimes \mathcal{O}_{\overline{X}}(\tilde{D}-D),\big\{K^{-\frac{1}{2}}\otimes\mathcal{O}_{\overline{X}}\big(-D_{2}(\alpha)+\tilde{D}-D\big)\big\}_{\alpha}\Big)\\
=&\deg K^{-\frac{1}{2}}\otimes \mathcal{O}_{\overline{X}}(\tilde{D}-D)+ \sum _{i=1}^{m}\frac{1}{2}(1+\{\theta\})\\
=&\frac{\chi(\overline{X})}{2}+\sum _{i=1}^{m}[\theta_{i}]-\frac{1}{2}m+\sum _{i=1}^{m}\frac{1}{2}\{\theta_{i}\}
\end{align*}
Hence, the parabolic Higgs stability condition
\[
\text{slope}\Big(K^{-\frac{1}{2}}\otimes \mathcal{O}_{\overline{X}}(\tilde{D}-D),\big\{K^{-\frac{1}{2}}\otimes\mathcal{O}_{\overline{X}}\big(-D_{2}(\alpha)+\tilde{D}-D\big)\big\}_{\alpha}\Big)<  \text{slope}(E,\{E_{\alpha}\})
\]
of $\big(E,\,\phi,\, \{E_\alpha\}  \big)$ reduces to
\[
\frac{\chi(\overline{X})}{2}+\sum _{i=1}^{m}[\theta_{i}]-\frac{1}{2}m+\sum _{i=1}^{m}\frac{1}{2}\{\theta_{i}\}<\frac{1}{2}\sum _{i=1}^{m}[\theta_{i}],
\]
which simplifies to
$
-\chi(\overline{X})+m-\sum _{i=1}^{m}\theta_{i}>0.
$
This coincides with the assumption $\chi (\overline{X})+\sum _{i=1}^{m}(\theta _{i}-1)<0$ in the lemma.

\end{proof}

\subsection{Proof of Theorem \ref{thm:rat}}
\label{subsec:thm1.1}

\quad In this subsection, we provide the proof of Theorem \ref{thm:rat} using the parabolic Higgs stable bundle $\big(E,\{E_\alpha\}\big)\to X$  constructed in Lemma \ref{lemma:construction of parabolic Higgs bundle}. In particular, we first construct an acceptable Hermitian metric (Definition \ref{definition:acceptable metric}) on the algebraic vector bundle $E\to X$. Then, applying Theorem \ref{theorem:Hermitian-Einstein metric}, we derive a Hermitian-Einstein metric that satisfies Hitchin's self-dual equation. This metric, in turn, induces a  hyperbolic metric that represents $\hat{D}$ on $\overline{X}$. This subsection builds upon the work in \cite{KW:2018}.

\begin{proof}[Proof of Theorem \ref{thm:rat}]
For each $p_i$ and its sufficiently small neighborhood $U\subset \overline{X}$, we first construct a local model metric on 
$E|_{U\setminus\{p_i\}}$, and then extend it to an acceptable Hermitian metric on $X$.

Recall that $E_{0}=K^{\frac{1}{2}}\oplus \Big(K^{-\frac{1}{2}}\otimes \mathcal{O}_{\overline{X}}(\tilde{D}-D)\Big)$ is a bundle over $\overline{X}$, and $E$ denotes its restriction to $X$.  Recall that $\omega$ is Euclidean near $p_i$. 
In a neighborhood $U$ of $p_{i}$, we can choose a complex coordinate $z$ around $p_i$ and a local holomorphic frame $e=(e_{1},e_{2})$ of the bundle $E_{0}$ such that $\omega=\frac{\rm i}{2}{\rm d}z\wedge {\rm d}\bar{z}$ and
$1_{\tilde{D}}|_{U}=z^{[\theta_{i}]}e^{*}_{1}\otimes e_{2}\cdot \frac{1}{z}{\rm d}z$. Thus, we can write $\phi$ as
\[
\phi=\left (\begin{array}{ccc}
0 & 0    \\
\frac{1}{2}z^{[\theta_{i}]} & 0   \\
\end{array}\right)z^{-1}{\rm d}z.
\]
We now express $({\rm d}z^{\frac{1}{2}},{\rm d}z^{-\frac{1}{2}})$ as a local holomorphic frame for $K^{\frac{1}{2}}\oplus K^{-\frac{1}{2}}$. Therefore, $e=(e_{1},e_{2})=({\rm d}z^{\frac{1}{2}},z^{1-[\theta_{i}]}{\rm d}z^{-\frac{1}{2}})$ forms a local holomorphic frame for $K^{\frac{1}{2}}\oplus\Big( K^{-\frac{1}{2}}\otimes\mathcal{O}_{\overline{X}}(\tilde{D}-D)\Big)$. 

Case 1.  If $\theta_{i}>0$, then we define the Hermitian metric $k_{\theta_{i}}$ on $E|_{U\setminus\{p_i\}}$ as
\begin{align*}
k_{\theta_{i}}(r)&=\left(
\begin{array}{ccc}
k_{\theta_{i}}(e_{1},e_{1}) & k_{\theta_{i}}(e_{1},e_{2})    \\
k_{\theta_{i}}(e_{2},e_{1}) & k_{\theta_{i}}(e_{2},e_{2})   \\
\end{array}\right)\\
&=\left(
\begin{array}{ccc}
e^{cr^{2}}\cdot\frac{r^{1-\{\theta_{i}\}}(1-r^{2\theta_{i}})}{2\theta_{i}} & 0   \\
0 & e^{cr^{2}}\cdot\frac{2\theta_{i}}{r^{-1-\{\theta_{i}\}}(1-r^{2\theta_{i}})}  \\
\end{array}\right)\\
&=\left(
\begin{array}{ccc}
e^{cr^{2}}\cdot\frac{r^{1+[\theta_{i}]}(r^{-\theta_{i}}-r^{\theta_{i}})}{2\theta_{i}} & 0   \\
0 & e^{cr^{2}}\cdot\frac{2\theta_{i}}{r^{-1+[\theta_{i}]}(r^{-\theta_{i}}-r^{\theta_{i}})}  \\
\end{array}\right)\\
&=\left(
\begin{array}{ccc}
-e^{cr^{2}}\cdot\frac{r^{1+[\theta_{i}]}}{\theta_{i}}\sinh (\theta_{i}\log r) & 0   \\
0 &  -e^{cr^{2}}\cdot\frac{\theta_{i}}{r^{-1+[\theta_{i}]}\sinh (\theta_{i}\log r)}  \\
\end{array}\right)\, .
\end{align*}
where $c=-\frac{\pi\sum _{i=1}^{m}[\theta_{i}]}{2\int_{X}\omega}$ and $|z|=r$. Then we have
\[
|e_{1}|_{k_{\theta_{i}}}\sim r^{\frac{1}{2}(1-\{\theta_{i}\})} \quad {\rm and}\quad |e_{2}|_{k_{\theta_{i}}}\sim r^{\frac{1}{2}(1+\{\theta_{i}\})}.
\]
The Hermitian adjoint of the Higgs field is
\[
\phi^{*k_{\theta_{i}}}=\frac{\theta_{i}^{2}}{z^{[\theta_{i}]}\sinh ^{2}(\theta_{i}\log r)}\left (\begin{array}{ccc}
0 & \frac{1}{2}   \\
0 & 0   \\
\end{array}\right)\bar{z}^{-1}{\rm d}{\bar{z}}.
\]
The curvature of $k_{\theta_{i}}$ is given by
\[
F_{k_{\theta_{i}}}=\bar{\partial}(k_{\theta_{i}}^{-1}\partial k_{\theta_{i}})=\left (\begin{array}{ccc}
-\frac{\theta_{i}^{2}}{4r^{2}\sinh ^{2}(\theta_{i}\log r)}+c & 0   \\
0 & \frac{\theta_{i}^{2}}{4r^{2}\sinh ^{2}(\theta_{i}\log r)}+c  \\
\end{array}\right){\rm d}\bar{z}\wedge {\rm d}z\, .
\]
The induced metric $\hat{k}_{\theta_{i}}$ on $\left(K^{\frac{1}{2}}\oplus K^{-\frac{1}{2}}\right)_{U\setminus\{p_i\}}$ is 
\begin{align*}
\hat{k}_{\theta_{i}}&=\left(
\begin{array}{ccc}
k_{\theta_{i}}(e_{1},e_{1}) & k_{\theta_{i}}(e_{1},{\rm d}z^{-\frac{1}{2}})    \\
k_{\theta_{i}}({\rm d}z^{-\frac{1}{2}},e_{1}) & k_{\theta_{i}}({\rm d}z^{-\frac{1}{2}},{\rm d}z^{-\frac{1}{2}})   \\
\end{array}\right)\\
&=\left(
\begin{array}{ccc}
-e^{cr^{2}}\cdot\frac{r^{1+[\theta_{i}]}}{\theta_{i}}\sinh (\theta_{i}\log r)  & 0   \\
0 &  -e^{cr^{2}}\cdot\frac{\theta_{i}}{r^{1-[\theta_{i}]}\sinh (\theta_{i}\log r)}  \\
\end{array}\right)\, .
\end{align*}
The induced Higgs field on $\left(K^{\frac{1}{2}}\oplus K^{-\frac{1}{2}}\right)_{U\setminus\{p_i\}}$
 is 
\[
\hat{\phi}=\left (\begin{array}{rrr}
0 & 0    \\
\frac{1}{2} & 0   \\
\end{array}\right){\rm d}z.
\]
The Hermitian adjoint of this Higgs field is
\[
\hat{\phi}^{*\hat{k}_{\theta_{i}}}=\frac{\theta_{i}^{2}}{r^{2}\cdot\sinh ^{2}(\theta_{i}\log r)}\left (\begin{array}{ccc}
0 & \frac{1}{2}   \\
0 & 0   \\
\end{array}\right){\rm d}{\bar{z}}.
\]
The curvature of $\hat{k}_{\theta_{i}}$ is
\[
F_{\hat{k}_{\theta_{i}}}=\left (\begin{array}{ccc}
-\frac{\theta_{i}^{2}}{4r^{2}\sinh ^{2}(\theta_{i}\log r)}+c & 0   \\
0 & \frac{\theta_{i}^{2}}{4r^{2}\sinh ^{2}(\theta_{i}\log r)}+c  \\
\end{array}\right){\rm d}\bar{z}\wedge {\rm d}z\, .
\]
Thus, $F_{\hat{k}_{\theta_{i}}}$ is acceptable and satisfies Hitchin's self-dual equation 
\[
F_{\hat{k}_{\theta_{i}}}+[\hat{\phi},\hat{\phi}^{*k_{\theta_{i}}}]=c\ {\rm Id}_{K^{\frac{1}{2}}\oplus K^{-\frac{1}{2}}}{\rm d}\bar{z}\wedge {\rm d}z
\]
on $\left(K^{\frac{1}{2}}\oplus K^{-\frac{1}{2}}\right)_{U\setminus\{p_i\}}$.

Case 2.  If $\theta_{i}=0$, then we define the Hermitian metric on $E|_{U\setminus\{p_i\}}$ as
\[
k_{0}(r)=\left(
\begin{array}{ccc}
 -e^{cr^{2}}\cdot r\log r & 0   \\
0 & -e^{cr^{2}}\cdot\frac{r}{\log r}   \\
\end{array}\right)\\
\]
where $|z|=r$. Then we have 
\[
|e_{1}|_{k_{0}}\leq  Cr^{\frac{1}{2}-\epsilon}, \quad |e_{2}|_{k_{0}}\leq  Cr^{\frac{1}{2}-\epsilon}
\]
for all $\epsilon>0$.
Note that the Higgs field is
\[
\phi=\left (\begin{array}{ccc}
0 & 0    \\
\frac{1}{2} & 0   \\
\end{array}\right)z^{-1}{\rm d}z.
\]
The Hermitian adjoint of the Higgs field is
\[
\phi^{*k_{0}}=\frac{1}{(\log r)^{2}}\left (\begin{array}{ccc}
0 & \frac{1}{2}   \\
0 & 0   \\
\end{array}\right)\bar{z}^{-1}{\rm d}{\bar{z}}.
\]
The curvature of $k_{0}$ is
\[
F_{k_{0}}=\bar{\partial}(k_{0}^{-1}\partial k_{0})=\left (\begin{array}{ccc}
-\frac{1}{4r^{2}(\log r)^{2}}+c & 0   \\
0 & \frac{1}{4r^{2}(\log r)^{2}}+c  \\
\end{array}\right){\rm d}\bar{z}\wedge {\rm d}z.
\]
The induced metric $\hat{k}_{0}$ on $K^{\frac{1}{2}}\oplus K^{-\frac{1}{2}}$ is 
\begin{align*}
\hat{k}_{0}&=\left(
\begin{array}{ccc}
k_{0}(e_{1},e_{1}) & k_{0}(e_{1},{\rm d}z^{-\frac{1}{2}})    \\
k_{0}({\rm d}z^{-\frac{1}{2}},e_{1}) & k_{0}({\rm d}z^{-\frac{1}{2}},{\rm d}z^{-\frac{1}{2}})   \\
\end{array}\right)\\
&=\left(
\begin{array}{ccc} -e^{cr^{2}}\cdot r\log r  & 0   \\
0 &  -e^{cr^{2}}\cdot\frac{1}{r\log r}  \\
\end{array}\right)\, .
\end{align*}
The induced Higgs field is
\[
\hat{\phi}=\left (\begin{array}{rrr}
0 & 0    \\
\frac{1}{2} & 0   \\
\end{array}\right){\rm d}z.
\]
The Hermitian adjoint of the Higgs field is
\[
\hat{\phi}^{*\hat{k}_{0}}=\frac{1}{r^{2}(\log r)^{2}}\left (\begin{array}{ccc}
0 & \frac{1}{2}   \\
0 & 0   \\
\end{array}\right){\rm d}{\bar{z}}.
\]
The curvature of $\hat{k}_{0}$ is
\[
F_{\hat{k}_{0}}=\left (\begin{array}{ccc}
-\frac{1}{4r^{2}(\log r)^{2}}+c & 0   \\
0 & \frac{1}{4r^{2}(\log r)^{2}}+c  \\
\end{array}\right){\rm d}\bar{z}\wedge {\rm d}z.
\]
Therefore, $F_{\hat{k}_{0}}$ is acceptable and satisfies Hitchin's self-dual equation
\[
F_{\hat{k}_{0}}+[\hat{\phi},\hat{\phi}^{*\hat{k}_{0}}]=c\ {\rm Id}_{K^{\frac{1}{2}}\oplus K^{-\frac{1}{2}}}{\rm d}\bar{z}\wedge {\rm d}z
\]
on $\left(K^{\frac{1}{2}}\oplus K^{-\frac{1}{2}}\right)_{U\setminus\{p_i\}}$.\\

From the constructions above, we have defined local model metrics in a neighborhood of each point $p_{i}$. Using these local metrics and partition of unity, we can obtain an acceptable Hermitian metric on $E$ over $X$. Since the filtered Higgs bundle $(E,\phi,\{E_{\alpha}\})$ is parabolic Higgs stable, Theorem \ref{theorem:Hermitian-Einstein metric} guarantees the existence of a Hermitian-Einstein metric $H$ on $E$ that satisfies the equation
\[
F_{H}+[\hat{\phi},\hat{\phi}^{*H}]=-2\pi {\rm i}\cdot \frac{\text{slope}(E,\{E_{\alpha}\})}{{\text{vol}}(X)}\ {\rm Id}_{E}\otimes \omega.
\]
Moreover, the Hermitian-Einstein metric $H$ is bounded  with respect to the model metric constructed above in a neighborhood of each $p_{i}$.

Note that $H$ is a metric on $K^{\frac{1}{2}}\oplus K^{-\frac{1}{2}}$. Assume that the matrix representation of $H$ under the holomorphic local frame $\left({\rm d}z^{\frac{1}{2}},\, {\rm d}z^{-\frac{1}{2}}\right)$ is
\[
H=\left (\begin{array}{ccc}
e^{v} & 0    \\
0 & e^{w}   \\
\end{array}\right).
\]
Then, $H$ induces a Riemannian metric $\widehat{{\rm d}s^{2}}=e^{w-v}|{{\rm d}z}|^{2}$ on the punctured Riemann surface, which is compatible with the Riemann surface structure.
Let
\[
\hat{H}=\left (\begin{array}{ccc}
e^{\frac{1}{2}(v-w)}  & 0    \\
0 & e^{\frac{1}{2}(w-v)}   \\
\end{array}\right)
\]
be the induced metric on $K^{\frac{1}{2}}\oplus K^{-\frac{1}{2}}$. Since $\hat{H}$ satisfies Hitchin's self-duality equations $F_{\hat{H}}+[\hat{\phi},\hat{\phi}^{*\hat{H}}]=0$, 
on the punctured surface $X$ it has constant Gaussian curvature $K_{\widehat{{\rm d}s^{2}}}\equiv -1$ (cf. Example 1.5 of \cite{Hitchin:1987}). Since an isolated singularity of a conformal hyperbolic metric can only be a cone singularity or a cusp singularity, it follows that $p_{i}$ is either a cone or a cusp singularity of $\widehat{{\rm d}s^{2}}$.

Case 1. If $\theta_{i}>0$, then $H$ is asymptotic to $\hat{k}_{\theta_{i}}$ near $p_{i}$. Then $\hat{H}$ is asymptotic to
\begin{align*}
\left(
\begin{array}{ccc}
-\frac{r}{\theta_{i}}\sinh (\theta_{i}\log r)  & 0   \\
0 &  -\frac{\theta_{i}}{r\sinh (\theta_{i}\log r)}  \\
\end{array}\right).
\end{align*}
Therefore, near $p_{i}$, the metric $\widehat{{\rm d}s^{2}}$ is asymptotic to
\begin{align*}
e^{w-v}&\sim \frac{\theta_{i}^{2}}{r^{2}\sinh ^{2}(\theta_{i}\log r)}=r^{2(\theta_{i}-1)}\left(\frac{4\theta^{2}_{i}}{(r^{2\theta_{i}}-1)^{2}}\right)
\\
&\sim r^{2(\theta_{i}-1)}.
\end{align*}
In this case, $\widehat{{\rm d}s^{2}}$ has a cone singularity at $p_{i}$ with cone angle $2\pi \theta_{i}$.

Case 2. If $\theta_{i}=0$, note that $\hat{H}$ is asymptotic to
\begin{align*}
\left(
\begin{array}{ccc} -r\log r  & 0   \\
0 &  -\frac{1}{r\log r}  \\
\end{array}\right).
\end{align*}
Hence, near $p_{i}$, the metric $\widehat{{\rm d}s^{2}}=e^{w-v}|{\rm d}z|^2$ is asymptotic to
$
\frac{|{\rm d}z|^2}{r^{2}(\log r)^{2}}.
$
In this case, $\widehat{{\rm d}s^{2}}$ has a cusp singularity at $p_{i}$.\\

The proof of uniqueness is straightforward, as discussed in \cite[Chapter \uppercase\expandafter{\romannumeral 2}]{Heins:1962}, and we present it here for completeness. Recall that $\mathrm{d} s_{0}^{2}$ is a smooth conformal metric on $\overline{X}$, which is Euclidean in a neighborhood of each point $p_{i}$. Let $\widehat{{\rm d}s^{2}}$ and ${\rm d}\sigma^{2}$
be two singular hyperbolic metrics representing the same divisor $\mathrm{\hat{D}}$. Assume that $\widehat{{\rm d}s^{2}}=e^{2\phi}\mathrm{d} s_{0}^{2}$ and ${\rm d}\sigma^{2}=e^{2\varphi}\mathrm{d} s_{0}^{2}$. Then, since the two singular hyperbolic metrics represent the same divisor, we have $\phi-\varphi\in C(\overline{X})\cap C^{\infty}(X)$.  By computation, we have $\Delta (\phi-\varphi)=f(e^{2\phi}-e^{2\varphi})$, where $f$ is a positive function on $X$. Define $\psi:=\max\{\phi-\varphi,0\}$. Then $\psi$ is subharmonic on $X$ and continuous on $\overline{X}$, so subharmonic on $\overline{X}$. Since $\overline{X}$ is compact, $\psi$ must be constant. Given that the curvatures of both $\widehat{{\rm d}s^{2}}$ and ${\rm d}\sigma^{2}$ are $-1$, this constant must be zero, which implies that $\phi\leq\varphi$. By symmetry, we conclude that $\phi=\varphi$.

\end{proof}

%% file: section-universal.tex
\section{Deformations of the Higgs field}
\label{section-universal}
\quad In this section, we deform the Higgs field $\phi$ as introduced in Lemma \ref{lemma:construction of parabolic Higgs bundle}. This deformation yields a family of Riemannian hyperbolic metrics, all possessing the same singularity type. This section builds on the work of \cite{Hitchin:1987, BGG:1997, BBDH:2021}. Throughout this section, we continue to use the notation of the previous sections.

Recall that $E_{0}=K^{\frac{1}{2}}\oplus \Big(K^{-\frac{1}{2}}\otimes \mathcal{O}_{\overline{X}}(\tilde{D}-D)\Big)$. Note that
\begin{align}
\label{align:e}
{\rm Hom} \Big(K^{-\frac{1}{2}}\otimes \mathcal{O}_{\overline{X}}(\tilde{D}-D),\ K^{\frac{1}{2}}\Big) \otimes K\otimes\mathcal{O}_{\overline{X}}(D)  \notag
\\=K^{2}\otimes\mathcal{O}_{\overline{X}}(2D-\tilde{D})
\subset {\rm End}(E_{0})\otimes K\otimes\mathcal{O}_{\overline{X}}(D).
\end{align}
and
\begin{align*}
&\deg {\rm Hom} \Big(K^{-\frac{1}{2}}\otimes \mathcal{O}_{\overline{X}}(\tilde{D}-D),\ K^{\frac{1}{2}}\Big)\otimes K=\deg K^{2}\otimes \mathcal{O}_{\overline{X}}(D-\tilde{D})\\
=&2(2g_{\overline{X}}-2)+m-\sum _{i=1}^{m}[\theta_{i}]
\geq 2g_{\overline{X}}-2+\left(2g_{\overline{X}}-2+m-\sum _{i=1}^{m}\theta_{i}\right)\\
>&2g_{\overline{X}}-2.
\end{align*}
We can embed $H^{0}\big(\overline{X},K^{2}\otimes\mathcal{O}_{\overline{X}}(D-\tilde{D})\big)$ into $H^{0}\big(K^{2}\otimes\mathcal{O}_{\overline{X}}(2D-\tilde{D})\big)$ via the mapping
$s\longmapsto s\otimes 1_{D}$. Therefore, using (\ref{align:e}), we obtain a natural embedding:
\begin{align}
\rho\colon H^{0}\big(\overline{X},K^{2}\otimes\mathcal{O}_{\overline{X}}(D-\tilde{D})\big) \longrightarrow H^{0}\big(\overline{X},{\rm End}(E_{0})\otimes K\otimes\mathcal{O}_{\overline{X}}(D)\big)
\end{align}
Note that the image of $\rho$ is contained in the image of the inclusion
\[
H^{0}\big(\overline{X},{\rm End}(E_{0})\otimes K\big)\longrightarrow H^{0}\big(\overline{X},{\rm End}(E_{0})\otimes K\otimes\mathcal{O}_{\overline{X}}(D)\big)
\]
With a slight abuse of notation, for any $a\in H^{0}\big(\overline{X},K^{2}\otimes\mathcal{O}_{\overline{X}}(D-\tilde{D})\big)$, we will also denote the corresponding element in $H^{0}\big(\overline{X},{\rm End}(E_{0})\otimes K\big)$ by $\rho(a)$.

In the following, we present the proof of the main theorem of this paper.

\begin{proof}[Proof of Theorem \ref{theorem:main theorem}]
Define a family of automorphisms of $E_{0}$ by
\[
T_\epsilon:=\left (\begin{array}{ccc}
\epsilon & 0    \\
0 &  1  \\
\end{array}\right)\quad {\rm where}\quad \epsilon>0.
\]
The parabolic Higgs bundle $(E_{0},\Phi_{a})$ is isomorphic to $(E_{0},T_\epsilon^{-1}\circ \Phi_{a}\circ T_\epsilon)$. Moreover,  $(E_{0},T_\epsilon^{-1}\circ \Phi_{a}\circ T_\epsilon)$ is parabolic Higgs stable if and only if so is $(E_{0},\Phi_{a})$. Since $\epsilon \neq 0$, we conclude that $(E_{0},T_\epsilon^{-1}\circ \Phi_{a}\circ T_\epsilon)$ is parabolic Higgs stable if and only if  so is $(E_{0},\frac{1}{\epsilon}T_\epsilon^{-1}\circ \Phi_{a}\circ T_\epsilon)$. By direct computation, we have 
\[
\frac{1}{\epsilon}T_\epsilon^{-1}\circ \Phi_{a}\circ T_\epsilon=\left (\begin{array}{ccc}
0 &  \frac{1}{\epsilon^{2}} a\otimes  1_{D}    \\
\frac{1}{2}1_{\tilde{D}} & 0   \\
\end{array}\right)=\Phi_{a/\epsilon^{2}}.
\]
By the openness of the stability condition, since $(E_{0},\phi)$ is stable, there exists a non-empty open set $V\subset H^{0}\big(\overline{X},K^{2}\otimes\mathcal{O}_{\overline{X}}(D-\tilde{D})\big)$ containing the origin, such that for every $a\in V$, the parabolic Higgs bundle $(E_{0},\Phi_{a})$ is parabolic stable. Choosing $\epsilon>0$ sufficiently large such that $\Phi_{a/\epsilon^{2}}\in V$, we conclude that $(E_{0},\Phi_{a})$ is parabolic stable for any $a\in H^{0}\big(\overline{X},K^{2}\otimes\mathcal{O}_{\overline{X}}(D-\tilde{D})\big)$.\\

1. The vector bundle $E_{0}$ is endowed with the Hermitian-Einstein metric $H_{a}$, which induces a metric $h_a$ on the line bundle $K$.  Using these two metrics, we construct a Hermitian metric on ${\rm End}(E_{0})\otimes K$. Since $\rho(a)\in H^{0}\big(\overline{X},{\rm End}(E_{0})\otimes K\big)$, we can compute its pointwise norm with respect to this metric.

Next, we will examine the behavior of $||\rho(a)||_{H_{a},h_{a}}$ near the punctures. First, observe that
\[
\rho(a)=\left (\begin{array}{ccc}
0 & a\otimes 1_{D}    \\
0 & 0   \\
\end{array}\right)\in H^{0}\big(\overline{X},{\rm End}(E_{0})\otimes K\otimes\mathcal{O}_{\overline{X}}(D)\big)
\]
and
\[
\Phi_{a}=\left (\begin{array}{ccc}
0 & a\otimes 1_{D}     \\
\frac{1}{2}\cdot 1_{\tilde{D}} & 0   \\
\end{array}\right),
\]

We consider $({\rm d}z^{\frac{1}{2}},{\rm d}z^{-\frac{1}{2}})$ as a local holomorphic frame of $K^{\frac{1}{2}}\oplus K^{-\frac{1}{2}}$. Then, the pair $e=(e_{1},e_{2})=({\rm d}z^{\frac{1}{2}},z^{1-[\theta_{i}]}{\rm d}z^{-\frac{1}{2}})$ forms a local holomorphic frame for the bundle $K^{\frac{1}{2}}\oplus\Big( K^{-\frac{1}{2}}\otimes\mathcal{O}_{\overline{X}}(\tilde{D}-D)\Big)$.
In a neighborhood $U$ of $p_{i}$, we choose a coordinate $z$ around $p_i$ such that $1_{\tilde{D}}|_{U}=z^{[\theta_{i}]}e^{*}_{1}\otimes e_{2}\cdot \frac{1}{z}{\rm d}z$ and $a=f(z)e_{2}^{*}\otimes e_{1}\cdot {\rm d}z$, where $f(z)$ is a holomorphic function on $U$. Thus, we have $a\otimes 1_{D}=zf(z)e_{2}^{*}\otimes e_{1}\cdot \frac{1}{z}{\rm d}z$ around $p_i$. We can now  express $\Phi_{a}$ and $\rho(a)$ as 
\[
\Phi_{a}=\left (\begin{array}{ccc}
0 &  zf(z)   \\
\frac{1}{2}z^{[\theta_{i}]} &  0   \\
\end{array}\right)z^{-1}{\rm d}z.
\]
and
\[
\rho(a)=\hat{\rho}(a){\rm d}z=\left (\begin{array}{ccc}
0 & f(z)   \\
0 &  0   \\
\end{array}\right){\rm d}z,
\]
Note that $\hat{\rho}(a)(e_{1})=0$ and $\hat{\rho}(a)(e_{2})=f(z)e_{1}$.
Thus, we have
\[
{\rm residue}\ (\Phi_{a})={\rm residue}\ (\phi).
\]
Therefore, by \cite[Sec.2]{Biquard:1997} and \cite[Sec.7]{Simpson:1990} (see also \cite[Definition 2.5]{KW:2018}), the two Hermitian metrics $H_{0}$ (corresponding to $a=0$) and $H_{a}$ on $E_{0}$ are mutually bounded, i.e., there exist positive constants $C_{1}$ and $C_{2}$ such that $C_{1}\cdot H_{0}\leq H_{a}\leq C_{2}\cdot H_{0}$.
Consequently, we have $||\rho(a)||_{H_{a},h_{a}}\sim ||\rho(a)||_{H_{0},h_{0}}$.

Case 1.  If $\theta_{i}=0$, then $||\rho(a)||_{H_{0},h_{0}}=||\hat{\rho}(a)||_{H_{0}}\cdot||{\rm d}z||_{h_{0}}\sim||\hat{\rho}(a)||_{k_{0}(r)}\cdot||{\rm d}z||_{h_{0}}=|f(z)|r|\log r|^{2}\leq C\cdot r|\log r|^{2}$. This implies that $||\rho(a)||$ tends to zero as we approach a puncture.

Case 2. If $\theta_{i}\neq 0$, then
\begin{align}
\label{align:local expression}
||\rho(a)||_{H_{0},h_{0}}=&||\hat{\rho}(a)||_{H_{0}}\cdot||{\rm d}z||_{h_{0}}\sim||\hat{\rho}(a)||_{k_{\theta_{i}}(r)}\cdot||{\rm d}z||_{h_{0}}  \notag \\
=&|f(z)|\frac{(1-r^{2\theta_{i}})^{2}}{4\theta_{i}^{2}}\cdot r^{1-\theta_{i}-\{\theta_{i}\}}
 \notag \\
=&|f(z)|\frac{(1-r^{2\theta_{i}})^{2}}{4\theta_{i}^{2}}\cdot r^{1-[\theta_{i}]-2\{\theta_{i}\}}.
\end{align}
Recall that, in the paragraph immediately preceding Theorem \ref{theorem:main theorem} in the introduction,  we already expressed the set of cone points as the following disjoint union 
\[
{\rm supp\, D}\setminus {\rm supp\, D_{0}}={\rm supp\, D_{1}} \cup {\rm supp\, D_{2}}\cup {\rm supp\, D_{3}}\cup {\rm supp\, D_{4}}.
\]

For any $x\in X$, choose a neighborhood $V\subset X$ and a local holomorphic coordinate $z$ such that $z(x)=0$. Then, $({\rm d}z^{\frac{1}{2}},{\rm d}z^{-\frac{1}{2}})$ is a local holomorphic frame for $E_{0}$. Assume that
$
\left (\begin{array}{ccc}
e^{v} & 0     \\
0 & e^{w}   \\
\end{array}\right)
$
is the matrix presentation of the Hermitian-Einstein metric $H_{a}$ with respect to $({\rm d}z^{\frac{1}{2}},{\rm d}z^{-\frac{1}{2}})$. The induced metric $\bar{H}_{a}$ has the matrix presentation
\[
\left (\begin{array}{ccc}
e^{\frac{1}{2}(v-w)} & 0     \\
0 & e^{\frac{1}{2}(w-v)}   \\
\end{array}\right),
\]
and we have $F_{\bar{H}_{a}}+[\Phi_{a},\Phi_{a}^{*\bar{H}_{a}}]=0$ on $K^{\frac{1}{2}}\oplus K^{-\frac{1}{2}}$ over $V$, where
\[
\Phi_{a}=\hat{\Phi}_{a}{\rm d}z=\left (\begin{array}{rrr}
0 & f(z)    \\
\frac{1}{2} & 0   \\
\end{array}\right){\rm d}z.
\]
Moreover, the induced K{\"a}hler metric on $X$ is $\hat{h}_{a}=g_{a}|{{\rm d}z}|^{2}=e^{w-v}|{{\rm d}z}|^{2}$, and the K{\"a}hler form is $\omega_{h_{a}}=\frac{\rm i}{2}e^{w-v}{\rm d}z\wedge {\rm d}\bar{z}$. By calculation, the self-duality equations become
\begin{equation}
\label{eqn:F_{1}}
F_{1}=(4||\rho(a)||^{2}_{H_{a},h_{a}}-1)({\rm i}\,\omega_{h_{a}}),
\end{equation}
where $F_{1}=\bar{\partial}\partial(v-w)$ is the curvature of the $U(1)$ connection on $K$.

In the neighborhood $V$, we can express $a=f(z){{\rm d}z}^{2}$, and
\[
\hat{h}_{a}=2f(z){{\rm d}z}^{2}+\left(g_{a}+4|f(z)|^{2}\cdot \frac{1}{g_{a}}\right){{\rm d}z}\otimes {{\rm d}\bar{z}}+2\overline{f(z)}{{\rm d}\bar{z}}^{2},
\]
which describes a real symmetric form on $TX$.
Therefore, $\hat{h}_{a}$ is a Riemannian metric on $X$ if and only if it is non-degenerate, i.e., if $\big(2||\rho(a)||_{H_{a},h_{a}}-1\big)^{2}>0$ for all $x\in X$. 
We will show that the norm of $\rho(a)$ is less than $\frac{1}{2}$  everywhere on $X$.


Next, suppose that
\[
0\neq a\in H^{0}\big(\overline{X},K^{2}\otimes\mathcal{O}_{\overline{X}}(D-2\tilde{D}-D_{3}-D_{4})\big).
\]
Then, by (\ref{align:local expression}), for any $\theta_{i}\neq 0$,   $||\rho(a)||$  tends to zero as we approach the punctures. Consequently, the infimum of the function $\frac{1}{2}-||\rho(a)||$ on $X$ must be attained at some point, say at $x_{0}\in X$.

Note that $\rho(a)$ can be viewed as a holomorphic quadratic differential on $X$, so ${\rm d}'' \rho(a)=0$. We have

\begin{align*}
{\rm d}''\langle {\rm d}' \rho(a), \rho(a)\rangle=&\langle {\rm d}''{\rm d}' \rho(a), \rho(a)\rangle-\langle {\rm d}' \rho(a), {\rm d}'\rho(a)\rangle \\
=& \langle F \rho(a), \rho(a)\rangle-\langle {\rm d}' \rho(a), {\rm d}'\rho(a)\rangle
\end{align*}
Using equation (\ref{eqn:F_{1}}), we obtain
\begin{align*}
\langle F \rho(a), \rho(a)\rangle=2(4||\rho(a)||^{2}_{H_{a},h_{a}}-1)({\rm i}\omega_{h_{a}})\langle \rho(a), \rho(a)\rangle
\end{align*}
Thus, we get
\begin{align*}
&{\rm d}''{\rm d}'\langle  \rho(a), \rho(a)\rangle={\rm d}''\langle {\rm d}'\rho(a), \rho(a)\rangle\\
=&2(1-4||\rho(a)||^{2})||\rho(a)||^{2}(-{\rm i}\omega_{h_{a}})-2||{\rm d}'\rho(a)||^{2}(-{\rm i} \omega_{h_{a}}),
\end{align*}
that is,
\[
\Delta ||\rho(a)||^{2}=2(1-4||\rho(a)||^{2})||\rho(a)||^{2}-2||{\rm d}' \rho(a)||^{2}
\]
where the Laplacian $\Delta$ is a positive operator. Since the operator $\mathcal{L}:=-\Delta-8||\rho(a)||^{2}$ is uniformly elliptic, we can apply the strong maximum principle \cite[Section VI.3., Proposition 3.3]{JT:1980} for the operator $\mathcal{L}$ and the point $x_{0}$, where $\mathcal{L}$ acts the function $\frac{1}{4}-||\rho(a)||^{2}$ on $X$.  We conclude that either $\frac{1}{2}-||\rho(a)||>0$ or $\frac{1}{2}-||\rho(a)||$ is a constant function. This proves that $\hat{h}_{a}$ is a Riemannian metric.\\

Now, assume that
$a\in H^{0}\big(\overline{X},K^{2}\otimes\mathcal{O}_{\overline{X}}(D-2\tilde{D}+D_{1}-D_{4})\big)$.
Note that the solution of the self-duality equation depends continuously on the initial data $a$. Furthermore, by (\ref{align:local expression}), for any $\theta_{i}\neq 0$,   $||\rho(a)||$ is bounded near $p_{i}$.  Thus, for any $a\in\mathcal{U}\subset H^{0}\big(\overline{X},K^{2}\otimes\mathcal{O}_{\overline{X}}(D-2\tilde{D}+D_{1}-D_{4})\big)$ in an appropriate open neighborhood $\mathcal{U}$ of $0$, the induced symmetric bilinear form $\hat{h}_{a}$ is non-degenerate by continuity, and therefore, it defines a Riemannian metric.\\

2. Note that $\hat{h}_{a}$ is a Riemannian metric, and $||\rho(a)||<\frac{1}{2}$ for all $x\in X$. From the computation in the proof of Theorem (11.2)(ii) in \cite[p. 120]{Hitchin:1987}, we conclude that $\hat{h}_{a}$ is a metric of curvature $-1$.

In a neighborhood of $x\in \overline{X}$, we can express $a=f(z){{\rm d}z}^{2}$, and
\begin{align*}
\hat{h}_{a}&=2f(z){{\rm d}z}^{2}+\left(g_{a}+4|f(z)|^{2}\cdot \frac{1}{g_{a}}\right){{\rm d}z}\otimes {{\rm d}\bar{z}}+2\overline{f(z)}{{\rm d}\bar{z}}^{2}\\
&=g_{a}\left({{\rm d}z}+\frac{1}{g_{a}}\cdot 2\overline{f(z)}{{\rm d}\bar{z}}\right)\left({{\rm d}\bar{z}}+\frac{1}{g_{a}}\cdot 2f(z){{\rm d}z}\right)
\end{align*}
We consider the measurable Beltrami differential $\beta$ locally given by
$\frac{2}{g_{a}}\cdot \frac{\overline{f(z)}{{\rm d}\bar{z}}}{{\rm d}z}$.
Note that $|\frac{2}{g_{a}}\cdot \overline{f(z)}|=||2\rho(a)||$, so
from the above, it follows that for $a\in \mathcal{U}$, we have
$||\beta||_{\infty}<1$.
By Proposition 4.8.12 in \cite{Hubbard:2006}, there exists a Riemann surface structure $\bar{\Sigma}$ on the compact topological surface underlying $\overline{X}$ such that the metric $\hat{h}_{a}$ on $\bar{\Sigma}\setminus {\rm supp\, D}=X$ lies in the conformal class determined by $\bar{\Sigma}$. Since an isolated singularity of a conformal hyperbolic metric is either a cusp or a cone singularity,  we conclude that $\hat{h}_{a}$ is a singular hyperbolic metric on $\bar{\Sigma}$ with singularities in ${\rm supp\, D}$.

It remains to show that $\hat{h}_{a}$ and $h_{0}$ have the same singularity type. First, recall the established fact that the Hermitian metrics $H_{0}$ and $H_{a}$ on $E_{0}$ are mutually bounded. This implies that the Riemannian metrics $h_{0}$ and $h_{a}$ on $X$ are also mutually bounded. In a neighborhood of $p_{i}$, we have
\begin{align}
\label{align:metric}
\hat{h}_{a}=g_{a}\left({{\rm d}z}+\frac{1}{g_{a}}\cdot 2\overline{f(z)}{{\rm d}\bar{z}}\right)\left({{\rm d}\bar{z}}+\frac{1}{g_{a}}\cdot 2f(z){{\rm d}z}\right)
\end{align}
where $z$ is the complex local coordinate of $\overline{X}$ in this neighborhood, and $a=f(z){{\rm d}z}^{2}$. Suppose that $w$ is the complex local coordinate near $p_{i}$ on $\bar{\Sigma}$. Since the metric $\hat{h}_{a}$ lies  in the conformal class of $\bar{\Sigma}$, we have
\begin{align}
\label{align:metric under new coordinate}
\hat{h}_{a}&=\rho(w)|{{\rm d}w}|^{2}\notag\\
&=\rho(z)|w_{z}|^{2} \left\lvert{{\rm d}z}+\frac{w_{\bar{z}}}{w_{z}}{{\rm d}\bar{z}}\right\rvert^{2}
\end{align}
Comparing the expressions (\ref{align:metric}) and (\ref{align:metric under new coordinate}), we see that $g_{a}=\rho(z)|w_{z}|^{2}$. Note that we can write $w$ as $w=\sum _{m,n\geq 0}c_{mn}z^{m}\bar{z}^{n}$,  where $c_{00}=0$ and $|c_{10}|>|c_{01}|$ by orientation-preserving considerations. Thus, as $z\rightarrow 0$, $|w_{z}|\sim c_{10}$, and $|w|$ and $|z|$ are mutually bounded. Since $h_{0}$ and $h_{a}$ are mutually bounded, we conclude that near $p_{i}$, $g_{a}\sim |z|^{2\theta_{i}-2}$ for $\theta_{i}>0$, and $g_{a}\sim |z|^{-2}(\log |z|)^{-2}$ for $\theta_{i}=0$ near $p_{i}$. Together with the fact that $\hat{h}_{a}$ is a singular hyperbolic metric on $\bar{\Sigma}$ with singularities in ${\rm supp\, D}$, we obtain the desired result.

\end{proof}

Below, we present the proofs of two corollaries of Theorem \ref{theorem:main theorem}.

\begin{proof}[Proof of Corollary \ref{corollary:restricted angles}]
Note that if we assume all $\theta_{i}\in[0,\frac{1}{2})$, then $D_{1}=D_{3}=D_{4}=0$. For any $a\in H^{0}\big(\overline{X},K^{2}\otimes\mathcal{O}_{\overline{X}}(D)\big)$, by (\ref{align:local expression}), for any $\theta_{i}$,   $||\rho(a)||$ convergences to zero as we approach the punctures. Thus, by applying the strong maximum principle, we conclude that $\hat{h}_{a}$ gives a Riemannian metric of constant Gaussian curvature $-1$. 
\end{proof}

\begin{proof}[Proof of Corollary \ref{corollary:isometric}]
 Let $\widetilde{{\rm d}s^{2}}$ be a Riemannian metric of constant curvature $-1$ on the $C^{\infty}$ surface $X$, which has cone singularities in ${\rm supp\, D}$ with cone angles in the range $(0,\pi]$. Additionally, $\widetilde{{\rm d}s^{2}}$ induces a compatible complex structure $M$ on $X$, which can be extended to $\overline{X}$, denoted by $\bar{M}$ (cf. \cite[ Section 2]{Judge:1995}).

According to Theorem 2 in \cite{Redman:2015}, there exists a unique harmonic diffeomorphism $u\colon X\rightarrow (M, \tilde{{\rm d}s^{2}})$ in the isotropy class (fixing the each point in ${\rm supp\, D}$) of the identity. Let $z$ denote a local isothermal coordinate on $X$, and assume that $\widetilde{{\rm d}s^{2}}=\sigma^{2}(u)|{\rm d}u|^{2}$.  We then have the following expression
\begin{align*}
u^{*}\widetilde{{\rm d}s^{2}}=&a+b+\bar{a}\\
=&q {\rm d}z^{2}+\sigma^{2}(u)(|\partial_{z}u|^{2}+|\bar{\partial}_{z}u|^{2})|{\rm d}z|^{2}+\bar{q} {\rm d}\bar{z}^{2},
\end{align*}
where $a$ is holomorphic on $X$ and meromorphic on $\overline{X}$, with at most simple poles at the points $p_{i}$ (cf. \cite[Section 5.1]{Redman:2015}).

Following Hitchin's proof, since $\widetilde{{\rm d}s^{2}}$ is a metric, we have $b^{2}-4a\bar{a}>0$ for all $x\in X$, and hence the root
\[
h_{a}=\frac{1}{2}\big(b+\sqrt{(b^{2}-4a\bar{a})}\big)
\]
to the equation $h_{a}+(a\bar{a}/h_{a})=b$ is everywhere positive and thus defines a metric on $X$ compatible with its complex structure.

Reversing the calculation of Theorem (11.2)(ii) in \cite[pg. 120]{Hitchin:1987} and using the fact that $u^{*}\widetilde{{\rm d}s^{2}}$ has constant curvature $-1$, we recover the self-duality equation (\ref{eqn:F_{1}}). Moreover, since $u^{*}\widetilde{{\rm d}s^{2}}$ shares the same type of singularities as $\widetilde{{\rm d}s^{2}}$ in ${\rm supp\, D}$ and $h_{a}\sim b$, the metrics $h_{a}$ and $\widetilde{{\rm d}s^{2}}$ are mutually bounded near the singularities.

Consequently, any conically singular metric of constant curvature $-1$ on $\overline{X}$ with cone singularities in ${\rm supp\, D}$ and cone angles in $(0,\pi]$ is isometric to a metric constructed from a solution to the self-duality equations.

\end{proof}

\begin{remark}
In the proof of Corollary \ref{corollary:isometric}, we rely on the existence of a unique harmonic diffeomorphism $u\colon X\rightarrow (M, \widetilde{{\rm d}s^{2}})$ in the isotropy class of the identity. However, this result requires that all cone angles of $\widetilde{{\rm d}s^{2}}$ lie in the interval $(0,\pi]$. This raises a natural question about the broader scenario:
Specifically, does there exist a unique harmonic diffeomorphism $u\colon X\rightarrow (M, \widetilde{{\rm d}s^{2}})$ in the isotropy class of the identity when $\widetilde{{\rm d}s^{2}}$ is a singular hyperbolic metric with cusp or cone singularities? If so, this would allow us to eliminate the cone angle restriction in Corollary \ref{corollary:isometric}.
\end{remark}